\documentclass[12pt, leqno]{amsart}
\usepackage{graphics,amsmath,amsthm,amssymb,graphicx}
\usepackage[enableskew]{youngtab}
\usepackage{tikz}
\usepackage[all,cmtip]{xy}
\usepackage{color}
\makeatletter
\newsavebox\myboxA
\newsavebox\myboxB
\newlength\mylenA

\newcommand\encircle[1]{%
  \tikz[baseline=(X.base)] 
    \node (X) [draw, shape=circle, inner sep=0] {\strut #1};}

\newcommand*\xoverline[2][0.75]{%
    \sbox{\myboxA}{$\m@th#2$}%
    \setbox\myboxB\null
    \ht\myboxB=\ht\myboxA%
    \dp\myboxB=\dp\myboxA%
    \wd\myboxB=#1\wd\myboxA
    \sbox\myboxB{$\m@th\overline{\copy\myboxB}$}
    \setlength\mylenA{\the\wd\myboxA}
    \addtolength\mylenA{-\the\wd\myboxB}%
    \ifdim\wd\myboxB<\wd\myboxA%
       \rlap{\hskip 0.5\mylenA\usebox\myboxB}{\usebox\myboxA}%
    \else
        \hskip -0.5\mylenA\rlap{\usebox\myboxA}{\hskip 0.5\mylenA\usebox\myboxB}%
    \fi}
\makeatother
\newtheorem{theorem}{Theorem}[section]

\newtheorem{proposition}[theorem]{Proposition}

\newtheorem{remark}[theorem]{Remark}
\newtheorem{lemma}[theorem]{Lemma}
\newtheorem{example}[theorem]{Example}

\begin{document}

\title[Maximal $(2k-1,2k+1)$-cores and $(2k-1,2k,2k+1)$-cores]{A Combinatorial proof of a Relationship Between Maximal $(2k-1,2k+1)$-cores and $(2k-1,2k,2k+1)$-cores}
\author[R. Nath]{Rishi Nath}
\address{Department of Mathematics} 
\address {York College, City University of New York, Jamaica, NY 11451}
\address{rnath@york.cuny.edu}
\author[J. A. Sellers]{James A. Sellers}
\address{Department of Mathematics, Penn State University, University Park, PA  16802}
\address {sellersj@psu.edu}


\date{\bf\today}

\begin{abstract}
Integer partitions which are simultaneously $t$--cores for distinct values of $t$ have attracted significant interest in recent years.  When $s$ and $t$ are relatively prime, Olsson and Stanton have determined the size of the maximal $(s,t)$-core $\kappa_{s,t}$.  When $k\geq 2$, a conjecture of Amdeberhan on the maximal $(2k-1,2k,2k+1)$-core $\kappa_{2k-1,2k,2k+1}$ has also recently been verified by numerous authors.

In this work, we analyze the relationship between maximal $(2k-1,2k+1)$-cores and maximal $(2k-1,2k,2k+1)$-cores. In previous work, the first author noted that, for all $k\geq 1,$
$$
\vert \, \kappa_{2k-1,2k+1}\, \vert = 4\vert \, \kappa_{2k-1,2k,2k+1}\, \vert
$$
and requested a combinatorial interpretation of this unexpected identity.  Here, using the theory of abaci, partition dissection, and elementary results relating triangular numbers and squares, we provide such a combinatorial proof. 
\end{abstract}

\maketitle


\bigskip


\section{Introduction}
\label{intro}
A {\it partition} $\lambda$ of the positive integer $n$ is a weakly decreasing sequence of positive integers which sum to $n.$  Each of the integers which make up the partition is known as a {\it part} of the partition.  For example, $(8,6,5,5,3,2,2,2,1)$ is a partition of the integer 34.  At times, we will employ an ``exponential'' notation when writing such partitions (in order to shorten the notation).  Thus, an alternative way to write the partition above, using this exponential notation, is 
$(8,6,5^2,3,2^3,1).$   

A {\it Young diagram} (or {\it Ferrers diagram}) is a pictorial representation of a partition.  Simply put, it is a finite collection of boxes which are arranged in left-justified rows with the row lengths weakly decreasing (since each row of the Young diagram corresponds to a part in the partition).  For example, the Young diagram of $(8,6,5^2,3,2^3,1)$  is given in Figure \ref{young_diagram_example}.
{\scriptsize
\begin{center}
\begin{figure}[h1]
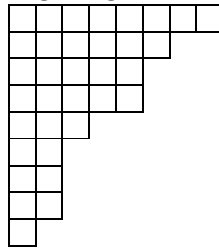

\label{young_diagram_example}
\hspace{-1cm}
\caption{Young diagram of  $(8,6,5^2,3,2^3,1)$}
\Yvcentermath1
$\yng(8,6,5,5,3,2,2,2,1)$
\end{figure}
\end{center}
} 
To each box in the Young diagram of $\lambda$ we assign a {\it hook}, which is the set of boxes in the same row and to the right, and in the same column and below, as well as the box itself.  The number of boxes is the {\it hook length} of that hook. The {\it first-column hook lengths} are those that appear in the left-most column of the Young diagram.  One obtains the {\it abacus diagram} of $\lambda$ by placing a bead on every value of ${\mathbb N}$ where a first-column hook length occurs. Positions without beads are called {\it spacers}. The {\it t-abacus} of $\lambda$ arises when the intervals $[0,t-1],[t,2t-1],[2t,3t-1],\cdots$ from the abacus of $\lambda$ are stacked on top of each other, forming $t$ runners.

A {\it t-core partition} (or simply $t$-core) of $n$ is a partition in which no hook of length $t$ appears in the Young diagram. A partition is a $t$-core if and only if its $t$-abacus has the property that no spacer occurs below a bead in any runner.  If a partition is not a $t$-core partition, one can obtain the {\it t-core of} $\lambda$ by removing a sequence of $t$-hooks until one is left with a partition $\lambda^{(t)}$ from which no $t$-hooks can be removed.  This is interpreted on the $t$-abacus of $\lambda$ by pushing down beads in each runner as far down as they can go. One then recovers the first-column hook lengths of $\lambda^{(t)}$ by unfolding the $t$-abacus back into an abacus and labeling zero as the first spacer.  This also shows that the $t$-core of $\lambda$ is unique.

When $t$ is prime, the $t$-cores label the $t$-defect zero blocks of $S_n,$ the symmetric group on $n$ letters.  Among other reasons, this makes the $t$-core partitions a significant set of combinatorial objects in the representation theory of the symmetric group.   

A partition $\lambda_{(\gamma)}$ can be read off of each runner $\gamma$ of  the $t$-abacus of $\lambda$ by considering each runner as its own abacus, marking the first spacer as zero and counting upwards. The positions which contain beads are then the first column hook lengths of $\lambda_{(\gamma)}.$ We call the sequence $(\lambda_{(0)},\cdots,\lambda_{(t-1)})$ the {\it t-quotient of} $\lambda$.

It is known (see, for example, \cite[Chapter 2]{O}) that the size of $\lambda,$ that is, the integer which $\lambda$ partitions, which we will denote by $\vert\, \lambda\,\vert,$ is simply the sum of the size of the $t$-core of $\lambda$ and $t$ times the size of all the partitions in the $t$-quotient of $\lambda$ (for any $t\geq 2$).   
\begin{equation}\label{corequo} |\lambda|=|\lambda^{(t)}|+\sum_{0\leq\gamma\leq t-1}|\lambda_{(\gamma)}|
\end{equation}

In this article, the 2-cores will play an especially important role.  These are the partitions of the form $\tau_j := (j, j-1, j-2, \dots, 2, 1)$ for some $j\geq 0.$   (We define $\tau_0 = \emptyset.$)   For each $j,$ the size of the partition $\tau_j$ is simply $T_j :=\frac{j(j+1)}{2},$ the $j^{th}$ triangular number.  

In recent years, the study of $t$--cores has expanded to include consideration of partitions which are simultaneously cores for various values of $t.$   The field began in 2002 when Anderson \cite{A} enumerated $(s,t)$-cores in the case when $s$ and $t$ are relatively prime. Subsequently, the work of Olsson and Stanton \cite{O-S} and others showed that, when gcd$(s,t)=1$, there is a unique $(s,t)$-core of largest size, denoted by $\kappa_{s,t}$, whose Young diagram contains the diagrams of all other $(s,t)$-cores. We call such a simultaneous core of largest size {\it maximal}.   
\begin{theorem} \label{st} Let gcd$(s,t)=1.$ Then there is a unique maximal $(s,t)$-core $\kappa_{s,t}$ which contains all others such that
\begin{equation*}
|\kappa_{s,t}|=\frac{(s^2-1)(t^2-1)}{24}.
\end{equation*}
\end{theorem}
In this note, we will focus our attention on two families of cores.  The first is the set of simultaneous $(2k-1,2k+1)$-cores of maximal size, which we denote by $\kappa_{s-1,s+1}$. Such cores were studied first by Amdeberhan and Leven \cite{AL}. From Theorem \ref{st}, we know 
\begin{equation}
\label{mainLHS}
\vert\, \kappa_{2k-1,2k+1}\,\vert = \frac{4k^2(k+1)(k-1)}{6}
\end{equation}
Furthermore, from analysis on $2k$-abaci done in \cite{Nath}, we have the following two remarks.
\begin{remark}
\label{rem1}
The $2k$-core of $\kappa_{2k-1,2k+1}$ is empty.  
\end{remark}

\begin{remark}
\label{rem2}
The $2k$-quotient structure of $\kappa_{2k-1,2k+1}$ consists solely of 2-cores; in particular, it is 
$$  \tau_{k-1}, \tau_{k-2}, \dots, \tau_1, \tau_0, \tau_0, \tau_1, \dots, \tau_{k-2}, \tau_{k-1}.$$
\end{remark}
See Figure 2 for a pictorial view of the $8$-quotient of $\kappa_{7,9}.$  

{\scriptsize
\begin{center}
\begin{figure}[h!]
\label{8quotient}
\hspace{-1cm}
\caption{8-quotient of $\kappa_{7,9}$}
\Yvcentermath1
$\yng(3,2,1)\;, \;\; \yng(2,1)\;,\;\; \yng(1)\;,\;\; \emptyset\;,\;\; \emptyset\;,\;\; \yng(1)\;,\; \;\yng(2,1)\;,\; \;\yng(3,2,1)$
\end{figure}
\end{center}
} 
Combining (\ref{corequo}) and (\ref{mainLHS}) with  Remarks \ref{rem1} and \ref{rem2}, we obtain another relation: 
$$\vert \, \kappa_{2k-1,2k+1}\,\vert = 2k\sum_{j=0}^{k-1} 2T_j=\frac{4k^2(k+1)(k-1)}{6}.$$

The second family of partitions we are interested in is simultaneous $(2k-1,2k,2k+1)$-cores of maximal size. Let $\kappa_{2k-1,2k,2k+1}$ be such a maximal $(2k-1,2k,2k+1)$-core. Then multiple authors, including Yang, Zhong, and Zhou \cite[Corollary 3.5]{Y-Z-Z}, Xiong \cite[Corollary 1.2]{X} and the first author \cite[$\S$5.1]{Nath} have verified that 
\begin{equation}
\label{mainRHS}
\vert\, \kappa_{2k-1,2k,2k+1}\,\vert = k\binom{k+1}{3} = \frac{k^2(k+1)(k-1)}{6},
\end{equation}
a result originally conjectured by Amdeberhan \cite{Am}. (Note: Unlike the $(2k-1,2k+1)$ case, the maximal $(2k-1,2k,2k+1)$-core is not self-conjugate; hence, there are two $(2k-1,2k,2k+1)$-cores of maximal size.
For the purposes of combinatorial manipulation, we choose $\kappa_{2k-1,2k,2k+1}$ to be the one {\it with the larger number of parts}.)

The first author \cite{Nath} recently noted that, thanks to (\ref{mainLHS}) and (\ref{mainRHS}), we have the following theorem.  

\begin{theorem}
\label{main_result}  
For any $k\geq 1,$ 
\begin{equation*}
\label{main_result_eqn}
\vert\, \kappa_{2k-1,2k+1}\,\vert =4\vert\, \kappa_{2k-1,2k,2k+1}\,\vert. 
\end{equation*}
\end{theorem}
In the same paper, the following question was asked: ``Is there an interpretation (either in the geometry of the $2k$-abacus or in the manipulation of Young diagrams) of the factor of 4 that appears above?''  
The goal of this note is to provide an affirmative answer to this question.  And, indeed, we will utilize both the geometry of the abaci in question as well as manipulation of Young diagrams in a natural way to prove our result.  

\section{A Combinatorial Proof}
In order to provide a purely combinatorial proof of Theorem \ref{main_result}, we wish to obtain a visual representation of  $\kappa_{2k-1,2k+1}$ and $\kappa_{2k-1,2k,2k+1}.$ We do this using the abacus. First we define $\alpha(2k)$ to be the $2k$-abacus of $\kappa_{2k-1,2k+1}$; from this construction arises the content of Remarks \ref{rem1} and \ref{rem2}.  To obtain a useful visual representation of a maximal $(2k-1,2k,2k+1)$-core, we consider $\bar{\alpha}(2k),$ the $2k$-abacus corresponding to the $(2k-1,2k,2k+1)$-core with the most parts. As a consequence of a result of Aggarwal \cite[Corollary 1.4]{Ag}, $\bar{\alpha}(2k)$ corresponds to $\kappa_{2k-1,2k,2k+1}$.  A longer discussion of this and a proof of the following result can be found in \cite[$\S$5.1]{Nath}.  
\begin{theorem}
\label{abacus-triple-catalan}  
The $2k$-abacus $\bar{\alpha}(2k)$ contains exactly $2k$ columns (called runners), which we label 0 to $2k-1$ (where runner 0 is the leftmost column) and exactly $k-1$ rows, which we label from 0 to $k-2$ (where row 0 is at the bottom of the abacus and row $k-2$ is at the top).  For each row $j,$ with $0\leq j\leq k-2,$ as we read from left to right, there are $j+1$ spacers, followed immediately by $2k-2(j+1)$ beads, followed immediately by $j+1$ spacers.  
\end{theorem}
Figures 3 and 4 are provided with the goal of clarifying the description in Theorem \ref{abacus-triple-catalan}.  

{\scriptsize
\begin{figure}[h!]
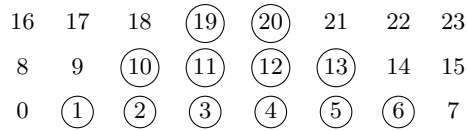

\label{abacus789}
\begin{center}

\[
\begin{array}{cccccccc}
{16} & 17 & 18 & \encircle{19} & \encircle{20} & 21 & 22 & 23\\
8 & {9} & \encircle{10} & \encircle{11} & \encircle{12} & \encircle{13} & 14 & 15 \\
0 & \encircle{1} & \encircle{2} & \encircle{3} & \encircle{4} & \encircle{5} & \encircle{6} & 7\\
\end{array} 
\] 

\caption{The $8$-abacus $\bar{\alpha}(8)$ of $\kappa_{7,8,9}$}

\end{center}
\end{figure}
} 
{\scriptsize
\begin{figure}[h!]
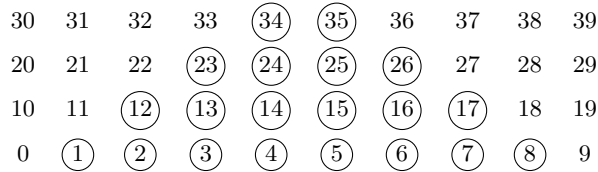

\label{abacus91011}
\begin{center}

\[
\begin{array}{cccccccccc}
{30} & 31 & 32 & 33 & \encircle{34} & \encircle{35} & 36 & 37 & 38 & 39 \\
{20} & 21 & 22 & \encircle{23} & \encircle{24} & \encircle{25} & \encircle{26} & 27 & 28 & 29 \\
10 & {11} & \encircle{12} & \encircle{13} & \encircle{14} & \encircle{15} & \encircle{16} & \encircle{17} & 18 & 19\\
0 & \encircle{1} & \encircle{2} & \encircle{3} & \encircle{4} & \encircle{5} & \encircle{6} & \encircle{7} & \encircle{8} & 9\\
\end{array} 
\] 

\caption{The $10$-abacus $\bar{\alpha}(10)$ of $\kappa_{9,10,11}$}

\end{center}
\end{figure}
} 

Thanks to Theorem \ref{abacus-triple-catalan} and the wonderful symmetries possessed by these abaci, we can easily prove a characterization of the Young diagram for $\kappa_{2k-1,2k,2k+1}.$ We will need the following result on triangular numbers.  
\begin{proposition}
\label{prop1}
The sum of two consecutive triangular numbers is a square.  That is, for any $n\geq 1,$ $T_n + T_{n-1} = n^2.$  
\end{proposition}
\begin{proof}
Of course, one way to prove this proposition is via straightforward algebraic manipulation of the equation.  However, in keeping with the combinatorial spirit of this work, we can see this result via a ``proof without words.'' For example, the statement that $T_5 + T_4 = 5^2$ van be visualized by the following figure:  

{\scriptsize
\begin{center}
\begin{figure}[h1]
\label{5by5}
\hspace{-1cm}
\caption{$T_5 + T_4 = 5^2$}
\Yvcentermath1
$\young(\bullet\bullet\bullet\bullet\bullet,\bullet\bullet\bullet\bullet *,\bullet\bullet\bullet **,\bullet\bullet ***,\bullet ****)$
\end{figure}
\end{center}
}
\end{proof}
We will also need the following lemma, which follows immediately from Theorem \ref{abacus-triple-catalan}.
\begin{lemma}\label{append} The $2k$-abacus $\bar{\alpha}(2k)$ is obtained from the $(2k-2)$-abacus $\bar{\alpha}(2k-2)$ by the following three steps:
\begin{enumerate}
\item Append two columns, one to the left and one to right of $\bar{\alpha}(2k-2)$, each consisting of $k-1$ spacers. 
\item Append a row below $\bar{\alpha}(2k-2)$, consisting of one spacer, then $2k-2$ consecutive beads, then one spacer.
\item Renumber the new $2k$-abacus starting with the bottom leftmost spacer being zero, and increasing values as one moves right, and then up. 
\end{enumerate}
\end{lemma}
\begin{example} The 10-abacus $\bar{\alpha}(10)$ is obtained from the 8-abacus $\bar{\alpha}(8)$ by appending two columns of four spacers each, to the left and right of 
$\bar{\alpha}(10)$, appending a row consisting of one spacer, eight beads, and one spacer below $\bar{\alpha}(10)$, and relabeling. See Figure 3 and Figure 4.
\end{example}
We now provide a characterization of the Young diagram of $\kappa_{2k-1,2k,k+1}.$
\begin{theorem}
\label{young-diagram-triple-catalan}  
When written as a partition in the ``exponential'' notation mentioned above, 
$$
\kappa_{2k-1,2k,2k+1} = \left(\left\{(k-1)^2\right\}^2, \left\{(k-2)^2\right\}^4, \dots, 16^{2k-8}, 9^{2k-6}, 4^{2k-4}, 1^{2k-2}\right).
$$
\end{theorem}
\begin{proof}
We prove this by induction on $k\geq2.$ When $k=2,$ the maximal $(3,4,5)$-core $\kappa_{3,4,5}$ equals $(1^2).$  

Suppose the result holds for $k-1$. By the inductive hypothesis, we know that 
\[
\kappa_{2k-3,2k-2,2k-1}=\left(\left\{(k-2)^2\right\}^2, \left\{(k-3)^2\right\}^4, \dots, 16^{2k-10}, 9^{2k-8}, 4^{2k-6}, 1^{2k-4}\right).
\]
By Lemma \ref{append}, we know we can move from $\bar{\alpha}(2k-2)$ to $\bar{\alpha}(2k)$ by adding a layer of ($2k-2$) beads below, adding the necessary spacers, and relabeling.  Then $\kappa_{2k-1,2k,2k+1}$ is comprised of the union of parts that arise from the beads of $\bar{\alpha}(2k)$ shifted up, which are now $\left(\left\{(k-1)^2\right\}^2, \left\{(k-3)^2\right\}^4, \dots, 16^{2k-8}, 9^{2k-6}, 4^{2k-4}\right)$ by Proposition \ref{prop1}, and $2k-2$ new parts of size 1 arising from the new row of beads.
\end{proof}
See Figure 6 for the Young diagram of $\kappa_{7,8,9}$ which follows from Theorem \ref{young-diagram-triple-catalan}.  
{\scriptsize
\begin{center}
\begin{figure}[h1]
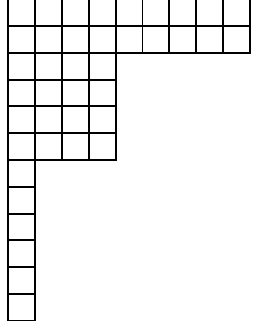

\label{young-diagram-789}
\hspace{-1cm}
\caption{Young diagram of  $\kappa_{7,8,9}$}
\Yvcentermath1
$\yng(9,9,4,4,4,4,1,1,1,1,1,1)$
\end{figure}
\end{center}
}

We are now in a position to develop a proof of Theorem \ref{main_result}.  However, rather than proving that theorem as it is written, we consider proving an equivalent form.  Namely, since the $2k$-quotient of $\kappa_{2k-1,2k+1}$ actually contains two copies of the sequence 
$$
 \tau_{k-1}, \tau_{k-2}, \dots, \tau_1, \tau_0,
$$
it makes sense to only speak of this first half of the quotient.  Thus, we define $Q_{2k-1,2k+1}$ as 
$$
Q_{2k-1,2k+1}:= \tau_{k-1}, \tau_{k-2}, \dots, \tau_1, \tau_0.
$$
Then, since $ \vert\, \kappa_{2k-1,2k+1}\,\vert $ equals $2k$ times the size of the corresponding $2k$-quotient, we see that Theorem \ref{main_result} is equivalent to 
$$
(2k)(2\vert\, Q_{2k-1,2k+1}\, \vert) = 4\vert\, \kappa_{2k-1,2k,2k+1}\,\vert, 
$$
which is equivalent to 
$$
k\vert\, Q_{2k-1,2k+1}\, \vert = \vert\, \kappa_{2k-1,2k,2k+1}\,\vert.
$$
Thus, our goal now is to prove this result, which we now write as its own theorem. 

\begin{theorem}
\label{main_result_modified}
For all $k\geq 1,$ 
$$
k\vert\, Q_{2k-1,2k+1}\, \vert = \vert\, \kappa_{2k-1,2k,2k+1}\,\vert.
$$
\end{theorem}

Our combinatorial proof will involve showing how the cells in $k$ copies of $Q_{2k-1,2k+1}$ can be placed in one--to--one correspondence with the cells of $\kappa_{2k-1,2k,2k+1}.$   
In order to complete such a proof, we will require the following proposition:  

\begin{proposition}
\label{prop2}
For any $n\geq 0,$ 
$$
3(T_1+T_2+\dots +T_n) = nT_{n+1}.
$$
\end{proposition}
\begin{proof}
Again, one can prove the above via induction on $n.$  However, in keeping with the combinatorial spirit, the interested reader is directed to the work of Zerger \cite{Zer} for a visual proof of this result.  
\end{proof}

We now possess all the tools necessary to prove Theorem \ref{main_result_modified}.  

\begin{proof} (of Theorem \ref{main_result_modified})
We begin with $k$ copies of $Q_{2k-1,2k+1},$ which can be seen (as a partition) as 
$$
(T_{k-1}^k, T_{k-2}^k, \dots, T_3^k, T_2^k, T_1^k).  
$$
We break this partition into three partitions as follows:  
$$
(T_{k-1}^{k-1})\cup (T_{k-1}, T_{k-2}, \dots, T_3, T_2, T_1) \cup (T_{k-2}^{k-1}, \dots, T_3^{k-1}, T_2^{k-1}, T_1^{k-1})
$$
We will refer to the three subpartitions mentioned above as Part 1, Part 2, and Part 3, respectively.  Now we show how each of these three parts correspond in a natural way to portions of the Young diagram of $\kappa_{2k-1,2k,2k+1}.$  

First, note that Part 3, which is by far the ``largest'' of the three parts, is actually made up of $k-1$ copies of $Q_{2(k-1)-1,2(k-1)+1}$ or $Q_{2k-3,2k-1}.$  Thus, by induction, Part 3 can be placed in one--to--one correspondence with $\kappa_{2k-3,2k-2,2k-1}.$  Note that, in the Young diagram, $\kappa_{2k-3,2k-2,2k-1}$ naturally lives inside $\kappa_{2k-1,2k,2k+1};$ indeed, $\kappa_{2k-1,2k,2k+1}$ contains a copy of $\kappa_{2k-3,2k-2,2k-1}$ along with two additional copies of parts of size $1, 4, 9, \dots, (k-1)^2.$  Thus, we now have to see how these two copies of each of the square parts correspond to Parts 1 and 2 mentioned above.  

Next, we consider Part 2 which is $ (T_{k-1}, T_{k-2}, \dots, T_3, T_2, T_1).$  These cells in $Q_{2k-1,2k+1}$ are naturally placed in one--to--one correspondence with the portion of the Young diagram of $\kappa_{2k-1,2k,2k+1}$ which consists of one copy of the largest square part, one copy of the third largest square part, one copy of the fifth largest part, and so on.  This correspondence is clear from Proposition \ref{prop1} above.  The correspondence is clear since each such square part from $\kappa_{2k-1,2k,2k+1}$ can be split into the sum of two consecutive triangular numbers, and since we have selected one copy of {\bf every other} square part from $\kappa_{2k-1,2k,2k+1},$ the split versions into triangular numbers will give us a collection of one copy of each triangular number.  And this is the structure of Part 2.  

Finally, we must consider Part 1 from the dissection of $Q_{2k-1,2k+1}.$  Note that Part 1 consists of exactly $k-1$ copies of the triangular number $T_{k-1}.$  In order for us to complete the one--to--one correspondence which will provide the proof of this theorem, it must be the case that the cells in Part 1 correspond to the cells that remain from $\kappa_{2k-1,2k,2k+1}. $   This entails one copy of the largest (square) part, two copies of the second largest part, one copy of the third largest part, two copies of the fourth largest part, and so on.  In order to see that these cells from the square parts correspond to Part 1, we again use Proposition \ref{prop1} to split each such square part into the sum of two consecutive triangular parts.  Thus, the (single) largest square part, which is of size $(k-1)^2,$ will be split into $T_{k-1} + T_{k-2}.$  The $T_{k-1}$ accounts for one of the $k-1$ parts of size $T_{k-1}$ in Part 1.  That leaves $k-2$ parts of size $T_{k-1}$ which are not yet accounted for.  But the remaining triangular pieces from the splitting of the squares mentioned above produce a sum of three copies of each triangular number (since we had one copy of the largest square part, two copies of the second largest part, one copy of the third largest part, etc.).  The number of cells here is given by 
$$
3(T_1+T_2+\dots +T_{k-2})
$$
and this quantity equals $(k-2)T_{k-1}$ thanks to Proposition \ref{prop2}.  This is exactly the same number of cells that we needed, and the proof is complete.  
\end{proof}
\section{Examples}
With the goal of making the verbal description of the mapping above much more clear, we show how the mapping works in two specific cases.  
\subsection{The case $k=4$}
The 4 copies of $Q_{7,9}$ can be seen as follows:  
{\scriptsize
\begin{center}
\begin{figure}[h1]
\Yvcentermath1
$\yng(3,2,1)\hspace{0.5in}\yng(2,1)\hspace{0.5in}\yng(1)$

\vspace{0.2in}

$\yng(3,2,1)\hspace{0.5in}\yng(2,1)\hspace{0.5in}\yng(1)$

\vspace{0.2in}

$\yng(3,2,1)\hspace{0.5in}\yng(2,1)\hspace{0.5in}\yng(1)$

\vspace{0.2in}

$\yng(3,2,1)\hspace{0.5in}\yng(2,1)\hspace{0.5in}\yng(1)$
\end{figure}
\end{center}
} 
The three ``Parts'' as described in the proof of Theorem \ref{main_result_modified}, that is, $T^3_3$, $(T_3,T_2,T_1)$, and $(T^3_2,T^3_1)$ can be seen as follows:  
{\scriptsize
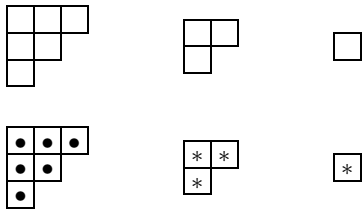
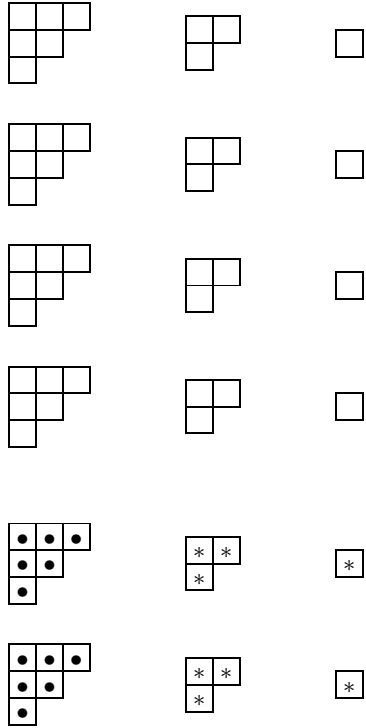
\begin{center}
\Yvcentermath1
$\yng(3,2,1)\hspace{0.5in}\yng(2,1)\hspace{0.5in}\yng(1)$

\vspace{0.2in}

$\young(\bullet\bullet\bullet,\bullet\bullet,\bullet)\hspace{0.5in}\young(**,*)\hspace{0.5in}\young(*)$

\vspace{0.2in}

$\young(\bullet\bullet\bullet,\bullet\bullet,\bullet)\hspace{0.5in}\young(**,*)\hspace{0.5in}\young(*)$

\vspace{0.2in}

$\young(\bullet\bullet\bullet,\bullet\bullet,\bullet)\hspace{0.5in}\young(**,*)\hspace{0.5in}\young(*)$

\end{center}
}
The cells filled with bullets {\scriptsize $\young(\bullet)$} make up Part 1, the empty cells {\scriptsize  $\yng(1)$} make up Part 2, and the cells filled with asterisks {\scriptsize  $\young(*)$} make up Part 3.  
Next, we show the Young diagram for $\kappa_{7,8,9}$ with the cells corresponding to Parts 1, 2, and 3 above labelled in similar fashion:  
{\scriptsize
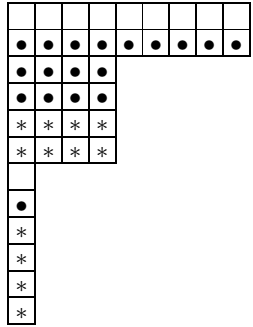
\begin{center}
\Yvcentermath1
$\young(\ \ \ \ \ \ \ \ \ ,\bullet\bullet\bullet\bullet\bullet\bullet\bullet\bullet\bullet,\bullet\bullet\bullet\bullet,\bullet\bullet\bullet\bullet,****,****,\ ,\bullet,*,*,*,*)$
\end{center}
}
\subsection{The case $k=5$}
The three ``Parts'' as described in the proof of Theorem \ref{main_result_modified} can be seen as follows:  
{\scriptsize
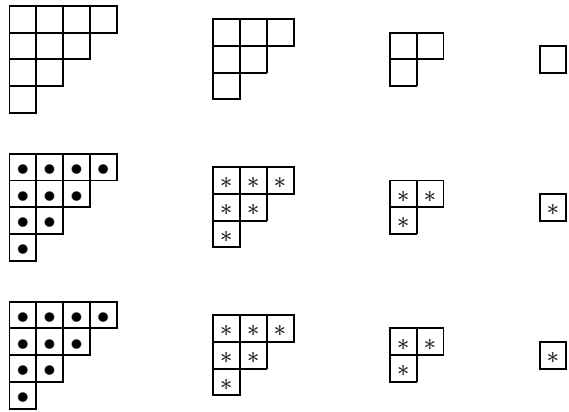
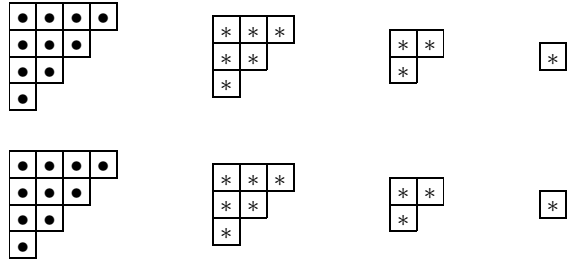
\begin{center}
\Yvcentermath1
$\yng(4,3,2,1)\hspace{0.5in}\yng(3,2,1)\hspace{0.5in}\yng(2,1)\hspace{0.5in}\yng(1)$

\vspace{0.2in}

$\young(\bullet\bullet\bullet\bullet,\bullet\bullet\bullet,\bullet\bullet,\bullet)\hspace{0.5in}\young(***,**,*)\hspace{0.5in}\young(**,*)\hspace{0.5in}\young(*)$

\vspace{0.2in}

$\young(\bullet\bullet\bullet\bullet,\bullet\bullet\bullet,\bullet\bullet,\bullet)\hspace{0.5in}\young(***,**,*)\hspace{0.5in}\young(**,*)\hspace{0.5in}\young(*)$

\vspace{0.2in}

$\young(\bullet\bullet\bullet\bullet,\bullet\bullet\bullet,\bullet\bullet,\bullet)\hspace{0.5in}\young(***,**,*)\hspace{0.5in}\young(**,*)\hspace{0.5in}\young(*)$

\vspace{0.2in}

$\young(\bullet\bullet\bullet\bullet,\bullet\bullet\bullet,\bullet\bullet,\bullet)\hspace{0.5in}\young(***,**,*)\hspace{0.5in}\young(**,*)\hspace{0.5in}\young(*)$
\end{center}
} 
\vspace{1.0in}
Next, we show the Young diagram for $\kappa_{9,10,11}$ with cells corresponding to Parts 1,2, and 3 labelled as above.
{\scriptsize
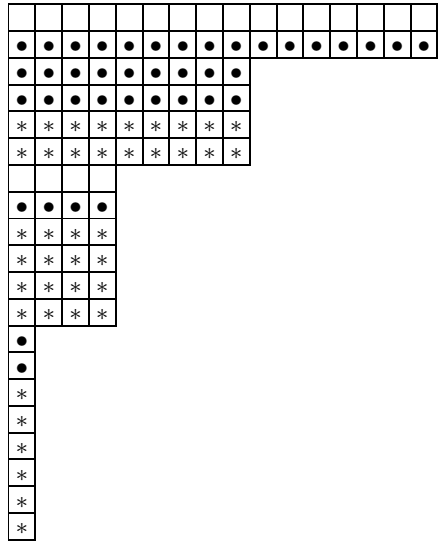
\begin{center}
\Yvcentermath1
$\young(\ \ \ \ \ \ \ \ \ \ \ \ \ \ \ \ ,\bullet\bullet\bullet\bullet\bullet\bullet\bullet\bullet\bullet\bullet\bullet\bullet\bullet\bullet\bullet\bullet,\bullet\bullet\bullet\bullet\bullet\bullet\bullet\bullet\bullet,\bullet\bullet\bullet\bullet\bullet\bullet\bullet\bullet\bullet,*********,*********,\ \ \ \ ,\bullet\bullet\bullet\bullet,****,****,****,****,\bullet,\bullet,*,*,*,*,*,*)$
\end{center}
}
\section{Closing Thoughts} 
We close by discussing some of the limiting factors in extending results such as Theorem \ref{main_result}.

First, note that similar divisibility cannot occur in the case of $(2k,2k+2)$-cores and $(2k,2k+1,2k+2)$-cores. For although Amdeberhan \cite{Am} has conjectured a value for $\vert\, \kappa_{2k,2k+1,2k+2} \,\vert$ (which has been verified by both Yang, Zhong, Zhou \cite{Y-Z-Z} and Xiong \cite{X}), a maximal $(2k,2k+2)$-core {\bf does not exist}. The lack of such a maximal core follows from the following more general result,  proved via generating functions in \cite{AKS} and from an abacus viewpoint in \cite{K}.
\begin{proposition} \label{nomax} Let gcd$(s,t)>1$. Then there are infinitely many simultaneous $(s,t)$-cores.
\end{proposition} 

There are also constraints on the kinds of simultaneous core partitions one can examine for more than two distinct integers. For example, we note the following: 

\begin{proposition} 
\label{propKX}
Let $s_1,s_2,s_3,\cdots$ be a sequence of positive integers. Then the number of simultaneous $(s_1,s_2,s_3,\cdots)$-cores is finite if and only if gcd$(s_1,s_2,s_3,\cdots)=1.$
\end{proposition}
A proof of Proposition \ref{propKX} can be found in \cite{K} and \cite{X}.
Even when gcd$(s_1,s_2,s_3)$=1, it is, in general, a hard problem in the theory of numerical semigroups to capture the size of the maximal $(s_1,s_2,s_3)$-core. However, as more information is gained about such sizes, the authors hope additional cases of divisibility will be discovered. 

\section{Acknowledgements}
The first author thanks George Andrews for supporting his visit to Pennsylvania State University in November 2014 where this research began. He also thanks Henry Cohn and Massachusetts Institute of Technology as a visitor in March and April 2015.  The first author was supported by grant PSC-CUNY TRADA-46-493 and a 2014-2015 Faculty Fellowship Leave from York College, City University of New York.

\end{document}